\documentclass{amsart}
\newtheorem{theorem}{Theorem}[section]
\newtheorem{lemma}[theorem]{Lemma}
\usepackage{color}
\theoremstyle{definition}

\theoremstyle{remark}

\numberwithin{equation}{section}

\begin{document}

\title{Weighted Sonine conditions and application}

\author{Xiangcheng Zheng}
\address{School of Mathematics, Shandong University, Jinan 250100, China}
\email{xzheng@sdu.edu.cn}
\thanks{This work was partially supported by the National Natural Science Foundation of China (No.12301555), the National Key R\&D Program of China (No. 2023YFA1008903), and the Taishan Scholars Program of Shandong Province (No. tsqn202306083).}

\author{Shangqin Zhu}
\address{School of Mathematics, Shandong University, Jinan 250100, China}
\email{shangqinzhu@163.com}

\author{Yiqun Li$^*$}
\address{Department of Mathematics, University of South Carolina, Columbia, SC 29208, USA}
\email{YiqunLi24@outlook.com}


\date{May 29, 2024.}

\dedicatory{(Communicated by Editor)}


\begin{abstract}
The Sonine kernel described by the classical Sonine condition of convolution form is an important class of kernels used in integral equations and nonlocal differential equations. This work extends this idea to introduce weighted Sonine conditions where the non-convolutional weight functions accommodate the inhomogeneity in practical applications. We characterize tight relations between classical Sonine condition and its weighted versions, which indicates that the non-degenerate weight functions may not introduce significant changes on the set of Sonine kernels. To demonstrate the application of weighted Sonine conditions, we employ them to derive equivalent but more feasible formulations of weighted integral equations and nonlocal differential equations to prove their well-posedness, and discuss possible application to corresponding partial differential equation models.
\end{abstract}

\maketitle


\section{Introduction}\label{Sect:Intro}
The Sonine kernel, which is initially  considered in \cite{Son}, represents a particular class of kernels that have been applied in both integral equations \cite{Car,Sam} and nonlocal differential equations \cite{Luc,LucFCAA}. For $b>0$, a kernel $k(t)\in L^1(0,b)$ is called a Sonine kernel if there exists another kernel $K(t)\in L^1(0,b)$ such that the  classical Sonine condition (denoted by CSC)
\begin{equation}\label{mh1}
\int_0^tK(t-s)k(s)ds=1
\end{equation}
holds for $t\in [0,b]$, where $K(t)$  is called a dual or associate kernel of $k(t)$ and  is  also a Sonine kernel. The Sonine kernel has favorable properties, e.g., the solutions to integral or nonlocal differential equations of convolution form with Sonine kernels could be analytically expressed with the assistance of the associate kernel, and there exist extensive investigations on the CSC \cite{Car,hanyga,Luc,LucFCAA,Luc2}.

However, as the convolution has the translation invariant feature, it is difficult to model the pointwise inhomogeneity. Instead, an additional non-convolutional weight function is usually involved to characterize such inhomogeneity. For instance, the weighted weakly-singular VIE with a power function convolution kernel is considered in many literature, see e.g. the books \cite{Bru,Gor} and the references therein. If we intend to consider the weighted problems with more general kernels, it is natural to propose weighted Sonine conditions.

Motivated by above discussions, we introduce weighted Sonine conditions (WSCs) with non-degenerate weight functions $w(s,t)$ (i.e. functions satisfying the condition (i) below), which contain two different types denoted by the WSC1 \eqref{mh2} and WSC2 \eqref{WSC2}. We first prove the following relations between the CSC and WSCs:
\begin{itemize}
\item[(A)] If $k$ satisfies WSC1, then its associated kernel $K$ satisfies CSC.

\item[(B)]  If $k$ satisfies WSC2, then $k$ satisfies CSC.

\item[(C)]  Converse to {(B)}, if a locally integrable completely monotone (LICM) $k$ satisfies CSC, then $k$ satisfies WSC2 with the same associated kernel.
\end{itemize}
From these relations we could observe that the CSC and WSCs have tight relations. In particular, (B) and (C) indicate that if $k$ satisfies WSC2, then $k$ satisfies CSC, and vice versa under an additional constraint, that is, the LICM condition. Nevertheless, the LICM condition is a widely used condition in studying the CSC \cite{hanyga} and covers many important Sonine kernels. Thus, these results suggest that the non-degenerate weight functions may not introduce significant changes on the set of Sonine kernels.

As a byproduct, we employ the WSCs to convert the weighted VIEs of the first kind and the nonlocal differential equations to more feasible formulations to perform mathematical analysis, which demonstrates the application of the derived WSCs. In particular, we establish the following well-posedness results:
\begin{itemize}
\item[(D)] The weighted first-kind VIE \eqref{nceq}  admits a unique $L^1$ solution.
\item[(E)] The nonlocal  differential equation \eqref{mh13}  admits a unique $L^1$ solution.
\end{itemize}

The rest of the paper is organized as follows: In \S \ref{wncsc}, we introduce two types of WSCs and incorporate the WSC1 to reformulate the first-kind VIE \eqref{mh23} into a more feasible form to carry out analysis. In particular, the statement (A) is proved as an application of the well-posedness of \eqref{mh23}. In \S \ref{IntDiff},  we employ the WSC1 to analyze the weighted first-kind VIE \eqref{nceq}  and the nonlocal differential equation \eqref{mh13}, which gives the statements (D) and (E). The tight connections between the CSC (\ref{mh1}) and  the WSC2 (\ref{WSC2}) as stated in (B) and (C) are proved in \S\ref{equivalence}. We finally address concluding remarks in the last section to show possible extensions for corresponding weighted nonlocal partial differential equations such as the weighted variable-exponent subdiffusion model.

Throughout this work,  we use $Q$ to denote a generic positive constant that may assume different values at different occurrences, and for a function $q(x_1,x_2)$ of two variables, we denote $q_i=\partial_{x_i}q$ for $i=1,2$.

\section{Weighted Sonine conditions}\label{wncsc}
Let the weight $w(s,t)$ satisfy
\begin{itemize}

\item[(i)] $w(t,t)\neq 0$ with $\mu_*\leq |w(t,t)|\leq \mu^*$ for $t\in [0,b]$ for some $0<\mu_*\leq \mu^*$;

\item[(ii)] $|w_2(s,t)|\leq \bar w(t-s)$ for $0\leq s\leq t\leq b$ for some $\bar w\in L^1(0,b)$.

\end{itemize}
Note that under these two conditions, $w$ is bounded as follows
$$|w(s,t)|=\Big|w(s,s)+\int_s^tw_2(s,\theta)d\theta\Big|\leq \mu^*+\|\bar w\|_{L^1(0,b)}. $$
Furthermore, $|w_2(s,t)|\leq Q$ or $|w_2(s,t)|\leq Q(t-s)^{-\sigma}$ for some $0<\sigma<1$ satisfies the condition (ii), which covers a wide class of functions.

Based on the conditions on $w$, we introduce two WSCs, namely WSC1 and WSC2. The WSC1 reads as follows:
for $k\in L^1(0,b)$, there exists a $K\in L^1(0,b)$ such that
\begin{equation}\label{mh2}
 \int_0^tw(s,z+s)K(t-z)k(z)dz=g(s,t),~~\forall 0\leq s+t\leq b
\end{equation}
for some function $g(s,t)$ satisfying
\begin{itemize}

\item[(a)] $g(s,0)\neq 0$ with $ \nu_*\leq |g(s,0)|\leq \nu^*$ for $0\leq s\leq b$ for some $0<\nu_*\leq \nu^*$;

\item[(b)]  $|g_2(s,t)|\leq  \bar g(t)$ for $0\leq s+t\leq b$ for some $\bar g\in L^1(0,b)$.

\end{itemize}
Under the relation (\ref{mh2}), $k(t)$ is called the weighted Sonine kernel  with   the associated kernel $K(t)$. When $w\equiv 1$, the left-hand side integral of (\ref{mh2}) is indeed a function of $t$ such that $g$ should simply be $g(t)$ on the right-hand side. In this case, the condition (\ref{mh2}) reduces to
\begin{equation*}
 \int_0^{t}K(t-z)k(z)dz=g(t),~~0\leq t\leq b.
\end{equation*}
If further $g(t)\equiv 1$, this condition degenerates to the CSC (\ref{mh1}).

We interchange the order of $K$ and $k$ in WSC1 \eqref{mh2} to introduce the WSC2 as follows: for $k\in L^1(0,b)$, there exists a $K\in L^1(0,b)$ such that
\begin{equation}\label{WSC2}
 \int_0^tw(s,z+s)k(t-z)K(z)dz=G(s,t),~~\forall 0\leq s+t\leq b
\end{equation}
for some function $G(s,t)$ satisfying conditions (a) and  (b), where we still denote its associated kernel   by $K$ for simplicity.
We note that the associated kernels in \eqref{mh2} and \eqref{WSC2} may refer to different functions.
\subsection{An example}
A typical example of the weighted Sonine kernel is the following variable-exponent Abel kernel
\begin{equation}\label{kt}
k(t)=t^{-\alpha(t)} \text{ for some }0<\alpha(t)<1\text{ on }[0,b],
\end{equation}
where the variable exponent $\alpha(t)$ describes, e.g., the variation of the memory and hereditary properties \cite{Lag,LiaSty,SunChaZha,ZenZha}. To simplify the derivations, we assume that
$$\alpha(t)\text{ is differentiable with }|\alpha'(t)|\leq L. $$
\begin{theorem}\label{wsc}
The $k(t)$ defined by (\ref{kt}) satisfies the WSC1 \eqref{mh2} with the associated kernel
\begin{equation*}
K(t)=\frac{t^{\alpha(0)-1}}{\kappa} \text{ with }\kappa=\Gamma(\alpha(0))\Gamma(1-\alpha(0)).
\end{equation*}
Here $\Gamma(\cdot)$ denotes the Gamma function.
\end{theorem}
\begin{proof}
It is clear that $K(t)\in L^1(0,b)$. To verify the conditions (a) and (b), direct calculation by the WSC1 \eqref{mh2} yields
\begin{align}
&\int_0^tw(s,z+s)K(t-z)k(z)dz\nonumber\\
&\quad=\frac{1}{\kappa}\int_0^tw(s,z+s)(t-z)^{\alpha(0)-1}z^{-\alpha(z)}dz\nonumber\\
&\quad=\frac{1}{\kappa}\int_0^1w(s,tz+s)(t-tz)^{\alpha(0)-1}(tz)^{-\alpha(tz)}tdz\label{mh7}\\
&\quad=\frac{1}{\kappa}\int_0^1w(s,tz+s)(tz)^{\alpha(0)-\alpha(tz)}(1-z)^{\alpha(0)-1}z^{-\alpha(0)}dz=:g(s,t).\nonumber
 \end{align}
We apply
$$\lim_{x\rightarrow 0^+}x^{\alpha(0)-\alpha(x)}=\lim_{x\rightarrow 0^+}e^{(\alpha(0)-\alpha(x))\ln x}=e^0=1 $$
and take $t=0$ in the right-hand side of (\ref{mh7}) to obtain
$$g(s,0)=\frac{w(s,s)}{\kappa}\int_0^1(1-z)^{\alpha(0)-1}z^{-\alpha(0)}dz=w(s,s),$$
 which, together with the property (i), implies the condition (a).

 To prove the condition (b), we differentiate $g$ in (\ref{mh7}) with respect to $t$ to obtain
 \begin{equation}\label{dg}
 \begin{array}{l}
 \displaystyle  g_2(s,t)=\frac{1}{\kappa}\int_0^1w_2(s,tz+s)z(tz)^{\alpha(0)-\alpha(tz)}(1-z)^{\alpha(0)-1}z^{-\alpha(0)}dz\\[0.1in]
 \displaystyle   + \frac{1}{\kappa}\int_0^1w(s,tz+s)\frac{d}{dt}\big[(tz)^{\alpha(0)-\alpha(tz)}\big](1-z)^{\alpha(0)-1}z^{-\alpha(0)}dz=:A_1+A_2.
 \end{array}
 \end{equation}
  By assumptions on $\alpha(t)$ and $x|\ln x|\leq Q$ for $x\in [0,b]$, we bound
\begin{equation}\label{xz2}
(tz)^{\alpha(0)-\alpha(tz)}=e^{(\alpha(0)-\alpha(tz))\ln(tz)}\leq e^{Ltz|\ln(tz)|}\leq e^{LQ}.
\end{equation}
 We use this and the property (ii) to bound $A_1$ in \eqref{dg} by
 $$B_1(t):=Q\int_0^1\bar w(tz)z(1-z)^{\alpha(0)-1}z^{-\alpha(0)}dz. $$
 To bound $A_2$ in \eqref{dg}, we apply properties (i) and (ii) to bound $w$ as
 \begin{align*}
 |w(s,tz+s)|&=\Big|w(s,s)+\int_s^{tz+s}  w_2(s,\theta) d\theta\Big|\\
 &\leq \mu^*+ \int_s^{tz+s}  \bar w(\theta-s) d\theta\leq \mu^*+\|\bar w\|_{L^1(0,b)}.
 \end{align*}
We then apply (\ref{xz2}) and the property of $\alpha(t)$ to bound
 \begin{align*}
 \Big|\frac{d}{dt}(tz)^{\alpha(0)-\alpha(tz)}\Big|&=\Big|(tz)^{\alpha(0)-\alpha(tz)}z\Big(-\alpha'(tz)\ln(tz)+\frac{\alpha(0)-\alpha(tz)}{tz}\Big) \Big|\\
 &\leq e^{LQ}z\big(L|\ln(tz)|+L\big)\leq Qz\big(|\ln(tz)|+1\big).
 \end{align*}
 We invoke the above two estimates and (\ref{xz2}) to bound $A_2$ in \eqref{dg}  as
 $$B_2(t):=Q\int_0^1 z\big(|\ln(tz)|+1\big) (1-z)^{\alpha(0)-1}z^{-\alpha(0)}dz. $$
 Consequently, the $g_2(s,t)$ could be bounded as
 $$|g_2(s,t)|\leq \bar g(t):=B_1(t)+B_2(t). $$

To show that $\bar g\in L^1(0,b)$,  direct calculation yields
\begin{align*}
\int_0^b|\bar g(t)|dt&=Q\int_0^1\int_0^b\big[\bar w(tz)z+z(|\ln(tz)|+1) \big]  dt(1-z)^{\alpha(0)-1}z^{-\alpha(0)}dz\\
&\leq Q\int_0^1\Big(\|\bar w\|_{L^1(0,b)}+\int_0^b t^{-1/2}(tz)^{1/2}(|\ln(tz)|+1)  dt\Big)\\
&\qquad\times(1-z)^{\alpha(0)-1}z^{-\alpha(0)}dz\\
&\leq Q\int_0^1 (1-z)^{\alpha(0)-1}z^{-\alpha(0)}dz\leq Q.
\end{align*}
Thus we complete the proof.
\end{proof}
\subsection{A first-kind VIE under WSC1}\label{Aux}
We analyze the following first-kind VIE
\begin{equation}\label{mh23}
	\int_0^tK(t-s)u(s)ds=f(t)
\end{equation}
where $K(t)$ is the associated kernel  defined by WSC1 \eqref{mh2} for some weighted Sonine kernel $k$.

\begin{theorem}\label{thm3}
Under  the WSC1 (\ref{mh2}), (\ref{mh23})  could be formally reformulated to the following second-kind VIE
\begin{equation}\label{mh24}
	u(t)+\frac{1}{g(0,0)}\int_0^t g_2(0,t-s)u(s)ds=\frac{1}{g(0,0)}\frac{d}{dt}\int_0^tw(0,s)k(s)f(t-s)ds.\end{equation}
\end{theorem}
\begin{proof}
We replace $t$ in (\ref{mh23}) by $y$, and then integrate  the resulting equation multiplied by $w(0,t-y)k(t-y)$ on both sides  from $0$ to $t$ to obtain	
\begin{equation}\label{mh25}
	\int_0^tw(0,t-y)k(t-y)\int_0^yK(y-s)u(s)dsdy=\int_0^tw(0,t-y)k(t-y)f(y)dy.
\end{equation}
We interchange the order of the integration on the left-hand side of the above equality and then apply the variable substitution $z=t-y$  to arrive at
\begin{equation}\label{yl2}\begin{array}{l}
\displaystyle \int_0^tw(0,t-y)k(t-y)\int_0^yK(y-s)u(s)dsdy \\[0.1in]
 \displaystyle \quad = \int_0^t\int_s^tw(0,t-y)K(y-s)k(t-y)dyu(s)ds\\[0.1in]
\displaystyle  \quad  =\int_0^t\int_0^{t-s}w(0,z)K(t-s-z)k(z)dzu(s)ds,
\end{array}
\end{equation}
which, combined with \eqref{mh25}, gives
\begin{equation*}
	\int_0^t\int_0^{t-s}w(0,z)K(t-s-z)k(z)dzu(s)ds=\int_0^{t}w(0,z)k(z)f(t-z)dz.
\end{equation*}
We combine  the WSC1 (\ref{mh2}) with $ t$ and $s$  replaced by $t-s$ and $0$, respectively, to reformulate the above equation as follows
\begin{equation}\label{mh26}
	\int_0^tg(0,t-s)u(s)ds=\int_0^{t}w(0,s)k(s)f(t-s)ds.
\end{equation}
We finally differentiate \eqref{mh26} with respect to $t$  to obtain
$$g(0,0)u(t)+\int_0^t g_2(0,t-s)u(s)ds=\frac{d}{dt}\int_0^tw(0,s)k(s)f(t-s)ds, $$
which, together with $g(0,0)\neq 0$ by the condition (a), leads to (\ref{mh24}) and thus completes the proof of the theorem.
\end{proof}

Compared with the original equation (\ref{mh23}),  the transformed equation (\ref{mh24}) exhibits a more feasible form that facilitates analysis. We then establish the equivalence between the original problem (\ref{mh23}) and the transformed equation (\ref{mh24}) in the following theorem.

\begin{theorem}\label{thmK}
Suppose $f$ is bounded. Then under the WSC1 (\ref{mh2}), an $L^1$ solution of the transformed equation (\ref{mh24}) is also a solution to the original equation (\ref{mh23}).
\end{theorem}
\begin{proof}
	Let $u(t)\in L^1(0,b)$ be a solution to the transformed equation (\ref{mh24}). Then we invoke \eqref{mh26} to rewrite (\ref{mh24}) as
\begin{equation}\label{yl5}
\frac{d}{dt}\Big(\int_0^t g(0,t-s)u(s)-w(0,s)k(s)f(t-s)ds\Big)=0,
\end{equation}
which implies
\begin{equation}\label{mh27}
\int_0^t g(0,t-s)u(s)-w(0,s)k(s)f(t-s)ds=c_0
\end{equation}
for some constant $c_0$.
We note that $u(t),\,k(t)\in L^1(0,b)$, $f(t)$ is bounded by assumption, $w$ is bounded by condition (i) and $g(0,t)$ is bounded by conditions (a) and (b) as follows
$$|g(0,t)|=\Big|\int_0^tg_2(0,y)dy+g(0,0)\Big|\leq \|\bar g\|_{L^1(0,b)}+\nu^*. $$
We apply these to  pass the limit $t\rightarrow 0^+$ in (\ref{mh27}) to get $c_0=0$, which, together with similar derivations as (\ref{mh25})--(\ref{mh26}), leads to
\begin{equation}\label{mh28}
\int_0^tw(0,t-y)k(t-y)\Big(\int_0^yK(y-s)u(s)ds-f(y)\Big)dy=0.
\end{equation}
Define
$$p(y):=\int_0^yK(y-s)u(s)ds-f(y) $$
such that (\ref{mh28}) could be written as
\begin{equation}\label{mh30}
	\int_0^tw(0,t-y)k(t-y)p(y)dy=0,
\end{equation}
and we remain to prove $p \equiv 0$.

We replace $t$  by $s$ in (\ref{mh30}), multiple $K(t-s)$ on both sides of the resulting equation, and then integrate it from $0$ to $t$ to obtain
\begin{equation*}
\int_0^tK(t-s)\int_0^sw(0,s-y)k(s-y)p(y)dyds=0.
\end{equation*}
We follow \eqref{yl2} to interchange the double integral and then apply the variable substitution $z=s-y$ to get
\begin{equation*}
\int_0^t\int_0^{t-y}w(0,z) K(t-y-z)k(z)dzp(y)dy=0,
\end{equation*}
which could be further reformulated by invoking  the WSC1 (\ref{mh2}) with $ t$ and $s$  replaced by $t-y$ and $0$, respectively, as follows
\begin{equation}\label{mh6}
\int_0^t g(0,t-y)p(y)dy=0.
\end{equation}
  Differentiate the above equation with respect to $t$ to get
\begin{equation}\label{yl6}
p(t)+\frac{1}{g(0,0)}\int_0^t g_2(0,t-y)p(y)dy=0,
\end{equation}
which, combined with condition (b), implies
$$
|p(t)|\leq\frac{1}{\nu_*} \int_0^t |g_2(0,t-y)||p(y)|dy.
$$
Apply the Gronwall inequality, see e.g., \cite[Lemma 2.7]{Lin} to the above inequality to obtain $p(t)\equiv 0$, i.e., $u$ solves the original equation (\ref{nceq}) and we thus complete the proof.
\end{proof}
\begin{theorem}\label{thmzz1}
Suppose $f\in W^{1,1}(0,b)$. Then the first-kind VIE (\ref{mh23}) admits a unique $L^1$ solution $u$.

In particular, taking $f=1$ in (\ref{mh23}) implies that $K$ satisfies the CSC with the associated kernel $u$, which leads to the statement (A).
\end{theorem}
\begin{proof}
If $f\in W^{1,1}(0,b)$, then the right-hand side function of (\ref{mh24}) belongs to $L^1(0,b)$. Then by classical results for the second-kind VIEs, see e.g., \cite[Theorem 2.3.5]{Gri}, the transformed equation (\ref{mh24}) admits a unique solution $u (t)\in L^1(0,b)$, which further ensures the existence and uniqueness of the solution $u (t)\in L^1(0,b)$ to the original integral equation  (\ref{mh23})   by Theorem \ref{thmK}.

Based on the well-posedness of  (\ref{mh23}), we take $f=1$ to find that $K(t)$ satisfies the CSC with its associated kernel $u$, which leads to the statement \textbf{I}(A)  in \S \ref{Sect:Intro}.
\end{proof}

\section{Integral and differential equations under WSC1}\label{IntDiff}
We demonstrate the application of the WSC1 in analyzing the weighted first-kind VIEs and nonlocal differential equations. We first consider the following weighted first-kind VIE \cite{Bru}
\begin{equation}\label{nceq}
\int_0^tw(s,t)k(t-s)u(s)ds=f(t),~~t\in (0,b],
\end{equation}
where $k\in L^1(0,b)$ satisfies the WSC1 \eqref{mh2} with the corresponding $g$.

\begin{theorem}\label{thm}
Under the WSC1 (\ref{mh2}), (\ref{nceq})  could be formally reformulated to the following second-kind VIE
\begin{equation}\label{mh10}
u(t)+\frac{1}{g(t,0)}\int_0^t g_2(s,t-s)u(s)ds=\frac{1}{g(t,0)}\frac{d}{dt}\int_0^tK(t-s)f(s)ds.
\end{equation}
\end{theorem}
\begin{proof}
We replace $t$ and $s$ in (\ref{nceq}) by $s$ and $y$, respectively, multiple $K(t-s)$ on both sides of the resulting equation, and then integrate from $0$ to $t$ to obtain
\begin{equation}\label{yl3}
\int_0^tK(t-s)\int_0^sw(y,s)k(s-y)u(y)dyds=\int_0^tK(t-s)f(s)ds.
\end{equation}
 We follow the similar procedures in (\ref{mh25})--(\ref{mh26}) to reformulate \eqref{yl3} as follows
\begin{equation}\label{yl4}
	\int_0^t g(s,t-s)u(s)ds=\int_0^tK(t-s)f(s)ds,
\end{equation}
and then  differentiate this equation with respect to $t$ on both sides to obtain
$$g(t,0)u(t)+\int_0^t g_2(s,t-s)u(s)ds=\frac{d}{dt}\int_0^tK(t-s)f(s)ds, $$
which, together with $g(t,0)\neq 0$ by condition (b), leads to (\ref{mh10}) and thus completes the proof.
\end{proof}

We then follow the proof of Theorem \ref{thmK} to show the equivalence between  (\ref{nceq}) and its reformulated version (\ref{mh10}).

\begin{theorem}\label{thm2}
Suppose $f$ is bounded. Under the WSC1 (\ref{mh2}), an $L^1$ solution of the transformed equation (\ref{mh10}) is also a solution to the original equation (\ref{nceq}).
\end{theorem}
\begin{proof}
Let $u(t)\in L^1(0,b)$ be the solution to the transformed equation (\ref{mh10}). Then we recall \eqref{yl4} to rewrite (\ref{mh10}) as
$$\frac{d}{dt}\Big(\int_0^t g(s,t-s)u(s)-K(t-s)f(s)ds\Big)=0, $$
which implies
$$\int_0^t g(s,t-s)u(s)-K(t-s)f(s)ds=0 $$
 by the analysis below (\ref{yl5}).
We then follow the  same derivations from  \eqref{yl3}--\eqref{yl4} to recover \eqref{yl3}, that is,
\begin{equation}\label{mh12}
\int_0^tK(t-s)\Big(\int_0^s\omega(y,s)k(s-y)u(y)dy-f(s)\Big)ds=0.
\end{equation}
Let
$$q(s):=\int_0^s\omega(y,s)k(s-y)u(y)dy-f(s) $$
such that (\ref{mh12}) could be written as
\begin{equation*}
\int_0^t K(t-s)q(s)ds=0.
\end{equation*}
 We apply Theorem \ref{thm3} to the above equation to get
\begin{equation*}
q(t)+\frac{1}{g(0,0)}\int_0^t g_2(0,t-s)q(s)ds=0,
\end{equation*}
and then follow the proof under \eqref{yl6} to obtain $q(t)\equiv 0$, i.e., $u$ solves the original equation (\ref{nceq}).\end{proof}

\begin{theorem}\label{thmxy1}
Suppose $f\in W^{1,1}(0,b)$ and $|g_2(s,t)|\leq Qt^{-\sigma}$ for some $0<\sigma<1$. Then (\ref{nceq}) admits a unique $L^1$ solution.
\end{theorem}
\begin{proof}
 The condition $f\in W^{1,1}(0,b)$ implies the right-hand side function of (\ref{mh10}) belongs to $L^1(0,b)$. Then we could follow the same technique as \cite[Theorem 2.4]{LiaSty} to show that the transformed equation (\ref{mh10}) admits a unique solution in $L^1(0,b)$ by means of the equivalent norm $\|q\|_{L^1_\lambda(0,b)}:=\int_0^b e^{-\lambda t}|q(t)|dt$. Then the well-posedness of the integral equation  (\ref{nceq})  follows from that of the transformed equation (\ref{mh10}) and Theorem \ref{thm2}.
\end{proof}
It is straightforward to extend the idea to investigate the following nonlocal differential equation, which has applications in various fields such as the viscoelasticity \cite{SamKil}
\begin{equation}\begin{array}{c}\label{mh13}
\displaystyle \frac{d}{dt}\int_0^tw(s,t)k(t-s)u(s)ds=f(t), ~~t\in (0,b],\\[0.1in]
 \displaystyle \int_0^t w(s,t)k(t-s)u(s)ds\Big|_{t=0}=c,
\end{array}
\end{equation}
where $c$ is a given constant and $k$ satisfies the WSC1 \eqref{mh2}.  In particular, the differential equation in \eqref{mh13} generalizes the fractional differential equations (with $w \equiv 1$ and $k = \frac{t^{-\beta}}{\Gamma(1-\beta)}$ for $ 0 < \beta<1 $) that have been studied in \cite{CarNgo,LuoHuy}.
\begin{theorem}\label{thm:diff}
Under the WSC1 (\ref{mh2}), the nonlocal differential equation (\ref{mh13}) is equivalent to the following second-kind VIE
\begin{equation}\label{mh14}
u(t)+\frac{1}{g(t,0)}\int_0^t g_2(y,t-y)u(y)dy=c\frac{K(t)}{g(t,0)}+\frac{1}{g(t,0)}\int_0^tK(t-y)f(y)dy.
\end{equation}

Suppose $f\in L^{1}(0,b)$ and $|g_2(s,t)|\leq Qt^{-\sigma}$ for some $0<\sigma<1$. Then (\ref{mh13}) admits a unique $L^1$ solution.
\end{theorem}
\begin{proof}
We first prove that \eqref{mh13} could be reformulated into \eqref{mh14}. Define $F(t):=\int_0^t f(s)ds + c$ and then integrate the first equation in (\ref{mh13}) from $0$ to $t$  to obtain
\begin{equation}\label{mh18}
	\int_0^tw(s,t)k(t-s)u(s)ds=F(t).
\end{equation}
Then we apply Theorem \ref{thm} to transform (\ref{mh18}) into the following second-kind
VIE
\begin{equation}\label{yl7}
	u(t)+\frac{1}{g(t,0)}\int_0^t g_2(y,t-y)u(y)dy= \frac{1}{g(t,0)}\frac{d}{dt}\int_0^tK(t-y)F(y)dy,
\end{equation}
where the right-hand side term could be further evaluated as
\begin{equation}\label{yl8}
	 \frac{1}{g(t,0)}\frac{d}{dt}\int_0^tK(t-y)F(y)dy=c\frac{K(t)}{g(t,0)}+\frac{1}{g(t,0)}\int_0^tK(t-y)f(y)dy,
\end{equation}
which gives \eqref{mh14}.

Conversely, suppose $u(t)\in L^1(0,b)$ is a solution to the transformed equation \eqref{mh14}. Then \eqref{yl7}--\eqref{yl8} and Theorem \ref{thm2} imply that $u$ solves the integral equation  \eqref{mh18}. As $F$ is differentiable, we differentiate \eqref{mh18} to arrive at the first equation in \eqref{mh13}. The initial condition could be obtained by taking the limits $t \rightarrow 0^+$ on both sides of \eqref{mh18}.

The second statement of this theorem could be proved similarly as Theorem \ref{thmxy1}. Thus we complete the proof.
\end{proof}

\section{Relation between CSC and  WSC2}\label{equivalence}
We prove tight connections between the CSC (\ref{mh1}) and  the WSC2 (\ref{WSC2}), which indicates that for non-degenerate weight functions, the CSC  and its weighted version are closely related to each other.
\begin{theorem}\label{thmequ}
	If a kernel $k(t)\in L^1(0,b)$ satisfies WSC2 (\ref{WSC2}), then it satisfies CSC (\ref{mh1}).
\end{theorem}
\begin{proof}
For a kernel $k(t)\in L^1(0,b)$ satisfying WSC2 (\ref{WSC2}), we aim at proving that there exists a $u \in L^1(0, b)$ such that
\begin{equation}\label{mh51}
\int _0^tk(t-s)u(s)ds=1.
\end{equation}

We multiple $w(0,t-s)K(t-s)$ on both sides of (\ref{mh51}) and follow the similar procedures in (\ref{mh25})--(\ref{yl2}) to arrive at
\begin{equation*}
	\int _0^t\int_0^{t-s}w(0,t-z-y)k(z)K(t-z-y)dzu(y)dy=\int_0^tw(0,t-s)K(t-s)ds,
\end{equation*}
which, combined with
\begin{equation*}
 \int_0^tw(s,t+s-z)k(z)K(t-z)dz=G(s,t)
\end{equation*}
derived from (\ref{WSC2}) by the variable substitution, implies
\begin{equation}\label{mh52}
	\int_0^tu(y)G(0,t-y)dy=\int_0^tw(0,t-s)K(t-s)ds.
\end{equation}
Differentiate (\ref{mh52})   with respect to $t$ on both sides to obtain the following second-kind VIE
\begin{equation}\label{SC:yli}
	u(t)G(0,0)+\int_0^tu(y)G_2(0,t-y)dy=\frac{d}{dt}\int_0^tw(0,t-s)K(t-s)ds.
\end{equation}
Based on  the properties of $G$ in the WSC2 (\ref{WSC2}) and the classical results on the second-kind VIE \cite[Theorem 2.3.5]{Gri}, \eqref{SC:yli} admits a unique solution $u(t)\in L^1(0,b)$.

Then we show that $u$ also solves (\ref{mh51}).
We employ \eqref{mh52} to rewrite (\ref{SC:yli}) as
\begin{equation*}
\frac{d}{dt}\Big(\int_0^tu(s)G(0,t-s)- w(0,t-s)K(t-s)ds\Big)=0,
\end{equation*}
which implies
\begin{equation*}
\int_0^tu(s)G(0,t-s)- w(0,t-s)K(t-s)ds =0
\end{equation*}
by the analysis  below \eqref{yl5}.
We then follow the  same derivations from  \eqref{mh51}--\eqref{mh52} to obtain
\begin{equation}\label{SC:yli2}
\int_0^tw(0,t-s)K(t-s)P(s)ds=0,
\end{equation}
where
$$P(s):=\int_0^s k(s-y)u(y)dy-1. $$
We multiple $k(t-s)$ on both sides of \eqref{SC:yli2} and follow the similar procedures in (\ref{mh25})--(\ref{yl2}) to obtain
\begin{equation*}
\int_0^t\int_0^{t-y}w(0,z) k(t-y-z)K(z)dzP(y)dy=0,
\end{equation*}
which could be further reformulated by invoking  the WSC2 (\ref{WSC2}) with $ t$ and $s$  replaced by $t-y$ and $0$, respectively, as follows
\begin{equation*}
\int_0^t G(0,t-y)P(y)dy=0.
\end{equation*}
We note that the above equation exhibits a similar form as \eqref{mh6} with $g$ and $p$ replaced by $G$ and $P$, respectively.
   We then follow the procedures below \eqref{mh6}  to obtain $P(t)\equiv 0$, i.e., $u$ solves the original equation (\ref{mh51}) and we thus complete the proof.
\end{proof}

We then follow  \cite{hanyga} to introduce the concept of LICM function to assist the proof of Theorem \ref{thmWSC2}, which will show that if $k$ satisfies CSC (\ref{mh1}) and is an LICM function, then $k$ satisfies WSC2  (\ref{WSC2}). By \cite{hanyga}, an LICM function is a locally integrable and infinitely differentiable function satisfying
$(-1)^n f^{(n)}(t)\geq 0$ for $n\geq 0$.

\begin{lemma}\label{lemk}
If an LICM function $k$ satisfies CSC (\ref{mh1}), then $k(t)<\infty $ for $t>0$.
\end{lemma}
\begin{proof}
By \cite[Theorem 4.1]{hanyga},  the associated kernel $K$ of $k$ is also an LICM function, and both $k$ and $K$ are unbounded near $t=0$.
 We  then intend to prove the lemma by contradiction. If not, we assume there exists a $t_0 \in (0, b]$ such that $k(t_0)=\infty $. By the property of the LICM   function, we have $k^\prime (t) \le 0$ on $(0, t_0) $, and thus $k(t) = \infty$  on $(0, t_0)$. Since $K(t)$  is a LICM function and  unbounded near $t=0$, there exists a $\hat t_0 \in (0, t_0)$ such that $K(t) \ge Q_0 $ on $(0, \hat t_0)$ for some $Q_0 > 0$.  We incorporate these to find that
$$\int_0^{\bar t} K(\bar t-s)k(s)ds \ge \int_0^{\bar t} Q_0 k(s) ds  =\infty, \quad \forall  \bar t \in (0, \hat t_0],$$
which contradicts to the CSC \eqref{mh1} and thus completes the proof.
\end{proof}

\begin{theorem}\label{thmWSC2}
If $k$ satisfies CSC (\ref{mh1}) and is a LICM function, then $k$ satisfies WSC2  (\ref{WSC2}) for weight functions $w$ satisfying condition (i) and $|w_2(s,t)|\leq Q$ for $0\leq s\leq t\leq b$ (a stronger version of the condition (ii)). Furthermore,
 the associated kernel $K$ of $k$ in CSC is also an associated kernel of $k$ in WSC2.
\end{theorem}
\begin{proof}
For an LICM function $k$ satisfying CSC (\ref{mh1}) with the associated kernel $K$ and a weight function $w$ satisfying condition (i) and $|w_2(s,t)|\leq Q$ for $0\leq s\leq t\leq b$, we intend to prove that $G(s,t)$ defined by (\ref{WSC2}) satisfies conditions $(a)$--$(b)$.

 For the sake of analysis, we invoke the CSC (\ref{mh1}) to reformulate $G(s, t)$ as follows
\begin{equation}\label{G}
	G(s,t)=w(s,s)+\int_0^t\big(w(s,t-z+s)-w(s,s)\big)k(z)K(t-z)dz.
\end{equation}
We incorporate  the CSC (\ref{mh1}) to obtain
\begin{equation*}
\begin{array}{l}
\displaystyle \Big|\int_0^t\big(w(s,t-z+s)-w(s,s)\big)k(z)K(t-z)dz\Big|\\[0.1in]
\displaystyle \quad  \leq Q\int_0^t(t-z)k(z)K(t-z)dz\\
\displaystyle \quad  \leq Qt\int_0^tk(z)K(t-z)dz = Qt \rightarrow 0 ~\mbox{as}~~ t \rightarrow 0^+,
\end{array}
\end{equation*}
which, combined with \eqref{G}, gives  $ G(s,0)=w(s,s)$ such that the condition (a) is satisfied.

To prove the condition $(b)$, we incorporate \eqref{G} to obtain
\begin{equation*}\begin{array}{l}
\displaystyle G_2(s,t)=\displaystyle\lim_{z\rightarrow t}\big(w(s,t-z+s)-w(s,s)\big)k(z)K(t-z) \\[0.125in]
\displaystyle \qquad \qquad +\int _0^tw_2(s,t-z+s)k(z)K(t-z)dz\\[0.125in]
\displaystyle \qquad \qquad  +\int_0^t\big(w(s,t-z+s)-w(s,s)\big)k(z)K'(t-z)dz=:C_1+C_2+C_3,
\end{array}
\end{equation*}
where $C_1$ could be bounded directly by the assumptions of the theorem
\begin{equation}\label{thmA1}
|C_1|\leq \lim_{z\rightarrow t}Q(t-z)k(z)K(t-z).
\end{equation}
 By \cite[Theorem 3.1]{hanyga}, we have $\displaystyle\lim_{t\rightarrow 0^+}tK(t)=0$.
We incorporate this with Lemma \ref{lemk} to conclude that the right-hand side term in \eqref{thmA1} and thus $C_1$ is $0$. By the assumptions of the theorem and the CSC (\ref{mh1}), we could bound $C_2$ in an analogous manner as follows
\begin{equation*}
|C_2|\leq Q\int_0^tk(z)K(t-z)dz=Q.
\end{equation*}
By \cite[Theorem 4.1]{hanyga}, the associated kernel $K$ of $k$ in CSC is also an LICM function with $K'(t)\leq 0$. We incorporate this property and the assumptions of the theorem to bound $C_3$ as follows
\begin{equation}\label{thmA3}
|C_3|\leq -Q\int_0^t(t-z)k(z)K'(t-z)dz=:D(t),
\end{equation}
and we then intend to prove $D(t)\in L^1(0,b)$. We invoke  the CSC (\ref{mh1}) and $\displaystyle\lim_{t\rightarrow 0^+}tK(t)=0$ to obtain
\begin{equation}\label{thmB}\begin{array}{rl}
 \displaystyle\int_0^b D(t)dt &\hspace{-0.1in}\displaystyle  =-Q\int_0^b\int_0^t(t-z)k(z)K'(t-z)dzdt\\[0.125in]
&\hspace{-0.1in} \displaystyle =-Q\int _0^bk(z)\int _z^b(t-z)K'(t-z)dtdz\\[0.125in]
&\hspace{-0.1in} \displaystyle =-Q\int _0^bk(z)\Big((b-z)K(b-z)-\int _z^bK(t-z)dt\Big)dz\\[0.125in]
&\hspace{-0.1in}\displaystyle  \leq Q \int _0^b(b-z)k(z)K(b-z)dz+Q\int _0^bk(z)\|K\|_{L^1(0,b)}dz\\[0.125in]
&\hspace{-0.1in}\displaystyle  \leq Qb+Q\|k\|_{L^1(0,b)}\|K\|_{L^1(0,b)}.
\end{array}
\end{equation}
 We combine \eqref{thmA1}--\eqref{thmA3} and  \eqref{thmB} to find that $|G_2(s,t)|\leq Q+ D(t) : = \bar G(t) \in L^1(0, b)$ such that the condition (b) holds, which completes the proof.
\end{proof}

\section{Concluding remarks}
This work extends the idea of the CSC to propose WSCs to accommodate the inhomogeneity in practical applications. We characterize tight relations between CSC and its weighted versions, and then employ WSCs to derive equivalent but more feasible formulations of weighted integral equations and nonlocal differential equations to prove their well-posedness. The derived results and tools provide a general instrument to treat the weighted nonlocal problems.

There are several potential extensions of the current work. For instance, one could utilize the transformed models of the weighted nonlocal problems to analyze the high-order solution regularity.
 A more substantial extension is to employ the WSCs to analyze the weighted subdiffusion of variable exponent, which models the anomalously diffusive transport through heterogeneous porous media
\begin{equation}\begin{array}{c}\label{PDEModel}
\displaystyle \int_0^t w(s,t)k(t-s)\partial_s u(\boldsymbol x, s)ds  -\Delta u(\boldsymbol x,t) = f(\boldsymbol x,t), \quad (\boldsymbol x, t) \in  \Omega \times (0, T],\\[0.05in]
\displaystyle   u(\boldsymbol x, 0) = u_0, \quad \boldsymbol x \in \Omega; \quad u(\boldsymbol x, t) = 0, \quad (\boldsymbol x, t) \in \partial \Omega \times [0, T].
  \end{array}
\end{equation}
Here $\Omega \subset \mathbb{R}^d$ $ (d = 1, 2, 3)$ is a bounded domain, $\boldsymbol x : = (x_1,  \cdots, x_d)^\top$,    and $k(t):=\frac{t^{-\alpha(t)}}{\Gamma(1-\alpha(t))}$ for some $0<\alpha(t)<1$. When $\alpha(t)\equiv\alpha$ for some $0<\alpha<1$ and $w\equiv 1$, model (\ref{PDEModel}) degenerates to the standard subdiffusion \cite{Lucsub,Sak}. When $\alpha(t)\equiv\alpha$ for some $0<\alpha<1$ and $w(s,t)=w(t-s)$, the nonlocal term takes the convolution form and model (\ref{PDEModel}) has been analyzed in \cite{Yam}. The general case that $w=w(s,t)$ has not been considered in the literature.

 Following the same procedure as the derivations in \S \ref{wsc}, one could show that $k$ satisfies the WSC1 \eqref{mh2} with the associated kernel $K(t):=\frac{t^{\alpha(0)-1}}{\Gamma(\alpha(0))}$.
Then one could apply the transformation in \S \ref{IntDiff} to convert \eqref{PDEModel} as
\begin{equation*}\begin{array}{c}\label{thm:PDE:e1}
\displaystyle  \partial_t u(\boldsymbol x, t)+\frac{1}{g(t,0)}\int_0^t g_2(s,t-s)\partial_s u(\boldsymbol x, s)ds - \frac{1}{g(t,0)} \frac{d}{dt} \int_0^tK(t-s)\Delta u(\boldsymbol x, s)ds\\[0.05in]
\displaystyle \qquad \qquad  = \frac{1}{g(t,0)}\frac{d}{dt} \int_0^tK(t-s)f(\boldsymbol x, s)ds, \quad (\boldsymbol x, t) \in  \Omega \times (0, T],\\[0.125in]
\displaystyle   u(\boldsymbol x, 0) = u_0, \quad \boldsymbol x \in \Omega; \quad u(\boldsymbol x, t) = 0, \quad (\boldsymbol x, t) \in \partial \Omega \times [0, T].
\end{array}
\end{equation*}
This model has the same form as \cite[Equation 12]{Zhe}, except for a time-dependent factor $\frac{1}{g(t,0)}$ multiplying on the $\Delta u$. Then one may combine the methods of handling the low-order term (the term containing $g_2$) in \cite{Zhe} and the treatments for the time-dependent factor  in, e.g. \cite{JinLi}, to perform analysis.

Another challenging issue is to relax the LICM condition used in Theorem \ref{thmWSC2}. For instance, if one could show that it suffices to impose the LICM only near the $0$, then the LICM condition is significantly relaxed. This is probably a reasonable improvement since the behavior of the Sonine kernels near $0$ usually plays the critical role. We will investigate this interesting topic in the near future.

\bibliographystyle{amsplain}

\end{document}